\documentclass{amsart}
\usepackage{amssymb, amsmath, mathrsfs, verbatim, url, mathabx}
\usepackage[all,cmtip]{xy}

\makeatletter
\@namedef{subjclassname@2020}{%
  \textup{2020} Mathematics Subject Classification}
\makeatother

\theoremstyle{plain}
\newtheorem{theorem}{Theorem}[section]
\newtheorem{lemma}[theorem]{Lemma}
\newtheorem{proposition}[theorem]{Proposition}
\newtheorem{corollary}[theorem]{Corollary}

\theoremstyle{definition}

\newtheorem{question}[theorem]{Question}

\newcommand{\wc}{\!\!\downarrow\,}
\newcommand{\bS}{\mathbf{\Sigma}}
\newcommand{\bP}{\mathbf{\Pi}}
\newcommand{\bD}{\mathbf{\Delta}}
\newcommand{\bG}{\mathbf{\Gamma}}
\newcommand{\bGc}{\widecheck{\mathbf{\Gamma}}}
\newcommand{\Borel}{\mathsf{Borel}}
\newcommand{\VL}{\mathsf{V=L}}
\newcommand{\ZFC}{\mathsf{ZFC}}
\newcommand{\ZF}{\mathsf{ZF}}
\newcommand{\DC}{\mathsf{DC}}
\newcommand{\AD}{\mathsf{AD}}

\newcommand{\PP}{\mathcal{P}}
\newcommand{\UU}{\mathcal{U}}
\newcommand{\VV}{\mathcal{V}}
\newcommand{\RRR}{\mathbb{R}}
\newcommand{\QQQ}{\mathbb{Q}}

\begin{document}

\title{Every finite-dimensional analytic space is $\sigma$-homogeneous}

\author{Claudio Agostini}
\address{Institut f\"{u}r Diskrete Mathematik und Geometrie
\newline\indent Technische Universit\"{a}t Wien
\newline\indent  Wiedner Hauptstra\ss e 8–-10/104
\newline\indent 1040 Vienna, Austria}
\email{claudio.agostini@tuwien.ac.at}
\urladdr{http://www.dmg.tuwien.ac.at/agostini/}

\author{Andrea Medini}
\address{Institut f\"{u}r Diskrete Mathematik und Geometrie
\newline\indent Technische Universit\"{a}t Wien
\newline\indent  Wiedner Hauptstra\ss e 8–-10/104
\newline\indent 1040 Vienna, Austria}
\email{andrea.medini@tuwien.ac.at}
\urladdr{http://www.dmg.tuwien.ac.at/medini/}



\thanks{The first-listed author was supported by the FWF grant P 35655. The second-listed author was supported by the FWF grant P 35588. The authors are grateful to Zolt\'{a}n Vidny\'{a}nszky for useful discussions, in which he ``planted the seeds'' for the results of Section \ref{section_main}.}

\date{March 19, 2024}

\begin{abstract}
All spaces are assumed to be separable and metrizable. Building on work of van Engelen, Harrington, Michalewski and Ostrovsky, we obtain the following results:
\begin{itemize}
\item Every finite-dimensional analytic space is $\sigma$-homogeneous with analytic witnesses,
\item Every finite-dimensional analytic space is $\sigma$-homogeneous with pairwise disjoint $\bD^1_2$ witnesses.
\end{itemize}
Furthermore, the complexity of the witnesses is optimal in both of the above results. This completes the picture regarding $\sigma$-homogeneity in the finite-dimensional realm. It is an open problem whether every analytic space is $\sigma$-homogeneous. We also investigate finite unions of homogeneous spaces.
\end{abstract}

\maketitle

\tableofcontents

\section{Introduction}

Throughout this article, unless we specify otherwise, we will be working in the theory $\ZF+\DC$, that is, the usual axioms of Zermelo-Fraenkel (without the Axiom of Choice) plus the principle of Dependent Choices (see \cite[Section 2]{carroy_medini_muller_constructing} for more details and references). By \emph{space} we will always mean separable metrizable topological space.

A space $X$ is \emph{homogeneous} if for every $(x,y)\in X\times X$ there exists a homeomorphism $h:X\longrightarrow X$ such that $h(x)=y$. For example, using translations, it is easy to see that every topological group is homogeneous (as \cite[Corollary 3.6.6]{van_engelen_thesis} shows, the converse is not true, not even for zero-dimensional Borel spaces). Homogeneity is a classical notion in topology, which has been studied in depth (see for example the survey \cite{arhangelskii_van_mill_survey}).

Here, we will focus on a much less studied notion.\footnote{\,See however \cite{arhangelskii_van_mill_homogeneous} and \cite{arhangelskii_van_mill_groups}, \cite{van_engelen_van_mill} and \cite{van_mill_homogeneous}, where somewhat related questions are investigated.} We will say that a space $X$ is \emph{$\sigma$-homogeneous} if there exist homogeneous subspaces $X_n$ of $X$ for $n\in\omega$ such that $X=\bigcup_{n\in\omega}X_n$. When each $X_n$ is closed in $X$ (respectively $X_n\in\bS^1_1(X)$ or $X_n\in\bP^1_1(X)$), we will say that $X$ is \emph{$\sigma$-homogeneous with closed witnesses} (respectively \emph{with analytic witnesses} or \emph{with coanalytic witnesses}). Notice that the complexity of the witnesses is always relative to $X$, as opposed to absolute (see Section \ref{section_preliminaries} for the definitions). Further instances of $\sigma$-homogeneity with special witnesses will appear throughout the paper, but their intended meaning seems clear enough not to warrant explicit definitions.

The following is the result that sparked our interest in this topic (see \cite{ostrovsky_sigma}), while the subsequent three theorems are the main results of \cite{medini_vidnyanszky}. Notice that the complexity ``closed'' is optimal, as $\omega+1$ gives a trivial example of a space that is not $\sigma$-homogeneous with open witnesses.

\begin{theorem}[Ostrovsky]\label{theorem_ostrovsky}
Every zero-dimensional Borel space is $\sigma$-homogeneous with pairwise disjoint closed witnesses.
\end{theorem}

\begin{theorem}[Medini, Vidny\'anszky]\label{theorem_determinacy}
Assume $\AD$. Then every zero-dimensional space is $\sigma$-homogeneous with pairwise disjoint closed witnesses.
\end{theorem}

\begin{theorem}[Medini, Vidny\'anszky]\label{theorem_counterexample}
In $\ZFC$, there exists a zero-dimensional space that is not $\sigma$-homogeneous.
\end{theorem}

\begin{theorem}[Medini, Vidny\'anszky]\label{theorem_definable_counterexample}
Assume $\VL$. Then there exists a coanalytic zero-dimensional space that is not $\sigma$-homogeneous.
\end{theorem}

It is clear from the above results that, in order to have a complete picture of $\sigma$-homogeneity in the zero-dimensional realm, only one piece is missing: is every zero-dimensional analytic space $\sigma$-homogeneous? This is \cite[Question 13.2]{medini_vidnyanszky}, and our main result gives an affirmative answer (see Theorem \ref{theorem_main}). The witnesses can be chosen analytic in general, or $\bD^1_2$ if one desires them to be pairwise disjoint (see Theorem \ref{theorem_main_disjoint}). In Section \ref{section_optimality}, we will see that these are the lowest-possible complexities. Furthermore, as Ostrovsky already noticed in \cite{ostrovsky_sigma}, results of this kind can easily be extended to all finite dimensions (see Section \ref{section_higher}). Finally, in Section \ref{section_finite}, we will examine the finite analogues of $\sigma$-homogeneity.

\section{Preliminaries, terminology and notation}\label{section_preliminaries}

Our reference for general set theory is \cite{jech}. Given a function $f:Z\longrightarrow W$, $A\subseteq Z$ and $B\subseteq W$, we will use the notation $f[A]=\{f(x):x\in A\}$ and $f^{-1}[B]=\{x\in Z:f(x)\in B\}$.

Our reference for topology is \cite{engelking}. We will write $X\approx Y$ to mean that the spaces $X$ and $Y$ are homeomorphic. A subset $M$ of a space $X$ is \emph{meager} if there exist nowhere dense subsets $M_n$ of $X$ for $n\in\omega$ such that $M\subseteq\bigcup_{n\in\omega}M_n$. A space $X$ is \emph{meager} if $X$ is a meager subset of $X$. A space is \emph{Baire} if no non-empty open subset of $X$ is meager. Given spaces $X$ and $Z$, an \emph{embedding} of $X$ in $Z$ is a function $j:X\longrightarrow Z$ such that $j:X\longrightarrow j[X]$ is a homeomorphism.

A zero-dimensional space $X$ is \emph{strongly homogeneous} (or \emph{h-homogeneous}) if every non-empty clopen subspace of $X$ is homeomorphic to $X$. This notion has been studied by several authors, both ``instrumentally'' and for its own sake (see the list of references in \cite{medini_products}). It is well-known that every strongly homogeneous space is homogeneous (see for example \cite[1.9.1]{van_engelen_thesis} or \cite[Proposition 3.32]{medini_thesis}). On the other hand, the reverse implication depends on set-theoretic assumptions (see \cite[Theorems 1.1 and 1.2]{carroy_medini_muller_homogeneous}).

Given a space $X$, recall that the \emph{Cantor-Bendixson derivative} of $X$ is defined as follows for every ordinal $\xi$:
\begin{itemize}
\item $X^{(0)}=X$,
\item $X^{(\xi+1)}=X^{(\xi)}\setminus\{x\in X^{(\xi)}:x\text{ is isolated in }X^{(\xi)}\}$,
\item $X^{(\xi)}=\bigcap_{\xi'<\xi}X^{(\xi')}$, if $\xi$ is a limit ordinal.
\end{itemize}
A space $X$ is \emph{scattered} if $X^{(\xi)}=\varnothing$ for some ordinal $\xi$. In this case, the minimal such $\xi$ is the \emph{Cantor-Bendixson rank} of $X$.

Our reference for descriptive set theory is \cite{kechris}. We will assume familiarity with the basic theory of Borel, analytic and coanalytic sets in Polish spaces. We will denote by $\Borel(X)$ the collection of all Borel subsets of a space $X$. Recall that, following \cite[page 315]{kechris}, one can define the classes $\bS^1_n(X)$, $\bP^1_n(X)$ and $\bD^1_n(X)=\bS^1_n(X)\cap\bP^1_n(X)$ for an arbitrary (that is, not necessarily Polish) space $X$, where $1\leq n<\omega$. Given $A\subseteq X$, simply declare $A\in\bS^1_n(X)$ (respectively $A\in\bP^1_n(X)$) if there exist a Polish space $Z$, an embedding $j:X\longrightarrow Z$, and $\widetilde{A}\in\bS^1_n(Z)$ (respectively $\widetilde{A}\in\bP^1_n(X)$) such that $j[A]=\widetilde{A}\cap j[X]$.

A space $X$ is a \emph{Borel space} (respectively \emph{analytic} or \emph{coanalytic}) if there exist a Polish space $Z$ and an embedding $j:X\longrightarrow Z$ such that $j[X]\in\Borel(Z)$ (respectively $j[X]\in\bS^1_1(Z)$ or $j[X]\in\bP^1_1(Z)$). It is well-known that a space $X$ is Borel (respectively analytic or coanalytic) iff $j[X]\in\Borel(Z)$ (respectively $j[X]\in\bS^1_1(Z)$ or $j[X]\in\bP^1_1(Z)$) for every Polish space $Z$ and every embedding $j:X\longrightarrow Z$ (see \cite[Proposition 4.2]{medini_zdomskyy}).

Next, we recall the fundamental definitions of Wadge theory, since it will be convenient to use this language in Section \ref{section_main} (see \cite{carroy_medini_muller_constructing} for more on this topic). Given a space $Z$ and $A,B\subseteq Z$, we will write $A\leq_\mathsf{W}B$ if there exists a continuous function $f:Z\longrightarrow Z$ such that $A=f^{-1}[B]$. Define $A\wc=\{B\subseteq Z:B\leq_\mathsf{W}A\}$. We will say that $\bG\subseteq\PP(Z)$ is a \emph{Wadge class} in $Z$ if there exists $A\subseteq Z$ such that $\bG=A\wc$. We will denote by $\bGc=\{Z\setminus A:A\in\bG\}$ the \emph{dual} class of $\bG$.

We conclude this section with a simple proposition, which can be safely assumed to be folklore. It shows that Baire spaces of sufficiently low complexity are Baire for a very strong reason.

\begin{lemma}\label{lemma_baire_dense_polish}
Let $X$ be a Baire space. Assume that $X$ is analytic or coanalytic. Then $X$ has a Polish dense subspace.
\end{lemma}
\begin{proof}
Assume without loss of generality that $X$ is a dense subspace of a Polish space $Z$. Since $X$ has the Baire property in $Z$, we can write $X=G\cup M$ by \cite[Proposition 8.23.ii]{kechris}, where $G\in\bP^0_2(Z)$ and $M$ is meager in $Z$. By \cite[Theorem 3.11]{kechris}, it will be enough to show that $G$ is dense in $Z$. Assume, in order to get a contradiction, that $U$ is a non-empty open subset of $Z$ such that $U\cap G=\varnothing$. Observe that $U\cap X$ is a non-empty open subset of $X$ because $X$ is dense in $Z$. Furthermore, using the density of $X$ again, it is easy to see that $M=M\cap X$ is meager in $X$ (see \cite[Exercise A.13.7]{van_mill_book}). Since $U\cap X\subseteq M$, this contradicts the assumption that $X$ is a Baire space.
\end{proof}

\section{Every zero-dimensional analytic space is $\sigma$-homogeneous}\label{section_main}

The following theorem combines \cite[Theorems 4.1(a) and 4.3(a)]{van_engelen}. Similar (if not identical) results were independently obtained by Ostrovsky (see \cite{ostrovsky_keldysh}) and Medvedev (see \cite{medvedev_meager} and \cite{medvedev_baire}).

\begin{theorem}[van Engelen]\label{theorem_van_engelen}
Let $X$ be a zero-dimensional space that satisfies the following conditions:
\begin{itemize}
\item Every non-empty clopen subspace of $X$ contains a closed subspace homeomorphic to $X$,
\item $X$ is either a meager space or it has a Polish dense subspace.
\end{itemize}
Then $X$ is strongly homogeneous.
\end{theorem}

Another result in the same family was obtained by Steel (see \cite[Theorem 2]{steel}), by making use of an ingenious lemma due to Harrington. We will not apply Steel's theorem, but instead we will combine Theorem \ref{theorem_van_engelen} with Harrington's lemma. This will make it easier to see that, for our purposes, no determinacy assumptions are necessary. This strategy was first employed by Michalewski in \cite{michalewski}. Before stating this lemma, we will need some preliminaries.

Given $i\in 2$, set
$$
Q_i=\{x\in 2^\omega:x(n)=i\text{ for all but finitely many }n\in\omega\}.
$$
Notice that every element of $2^\omega\setminus(Q_0\cup Q_1)$ is obtained by alternating finite blocks of zeros and finite blocks of ones. Define the function $\phi:2^\omega\setminus(Q_0\cup Q_1)\longrightarrow 2^\omega$ by setting
$$
\phi(x)(n)=\left\{
\begin{array}{ll} 0 & \textrm{if the $n^{\text{th}}$ block of zeros of $x$ has even length},\\
1 & \textrm{otherwise,}
\end{array}
\right.
$$
where we start counting with the $0^{\text{th}}$ block of zeros. Notice that $\phi$ is continuous. A Wadge class $\bG$ in $2^\omega$ is \emph{reasonably closed} if $\phi^{-1}[A]\cup Q_0\in\bG$ for every $A\in\bG$. For the purposes of this article, the only fact concerning these classes that we will need is that $\bS^1_1(2^\omega)$ and $\bP^1_1(2^\omega)$ are reasonably closed Wadge classes.

The original statement of Lemma \ref{lemma_harrington} (see \cite[Lemma 3]{steel}) asks that $B\in\bG\setminus\bGc$ and that suitable determinacy assumptions hold, from which $\bG=B\wc$ follows by Wadge Lemma (see \cite[Lemma 4.4]{carroy_medini_muller_constructing}). However, it is clear from the proof that this is the only place where the determinacy assumptions are used. Therefore, it is possible to give the following determinacy-free version of this result.

\begin{lemma}[Harrington]\label{lemma_harrington}
Let $\bG$ be a reasonably closed Wadge class in $2^\omega$, and let $A,B\subseteq 2^\omega$. Assume that $A\leq_{\mathsf{W}} B$ in $2^\omega$ and $\bG=B\wc$. Then there exists an injective continuous function $f:2^\omega\longrightarrow 2^\omega$ such that $A=f^{-1}[B]$.
\end{lemma}

Now we have all the tools needed to prove the zero-dimensional version of our main result (see Theorem \ref{theorem_main}). In fact, the core of our argument is contained in the proof of the following lemma. As in \cite{van_engelen_thesis}, we will say that a space $X$ is \emph{nowhere countable} if $X$ is non-empty and no non-empty open subset of $X$ is countable.

\begin{lemma}\label{lemma_main}
Let $X$ be a zero-dimensional analytic space. Assume that $X$ is nowhere countable, and that $X$ is either a meager space or a Baire space. Then $X$ is $\sigma$-homogeneous with analytic witnesses.
\end{lemma}
\begin{proof}
Let $\{U_n:n\in\omega\}$ be a base for $X$ consisting of non-empty clopen sets. Using the fact that analytic sets have the perfect set property, together with the assumption that $X$ is nowhere countable, it is a simple exercise to obtain $K_{n,i}$ for $(n,i)\in\omega\times 2$ satisfying the following conditions:
\begin{itemize}
\item Each $K_{n,i}\approx 2^\omega$,
\item Each $K_{n,i}$ is nowhere dense in $X$,
\item Each $K_{n,i}\subseteq U_n$,
\item $K_{n,i}\cap K_{m,j}=\varnothing$ whenever $(n,i)\neq (m,j)$.
\end{itemize}
By \cite[Theorem 26.1]{kechris}, we can fix $A\subseteq 2^\omega$ such that $A\wc=\bS^1_1(2^\omega)$. Fix homeomorphisms $h_{n,i}:2^\omega\longrightarrow K_{n,i}$ for $(n,i)\in\omega\times 2$, then set $A_{n,i}=h_{n,i}[A]$. Define
$$
X_i=\left(X\setminus\bigcup_{n\in\omega}K_{n,i}\right)\cup\left(\bigcup_{n\in\omega}A_{n,i}\right)
$$
for $i\in 2$. It is clear that each $X_i$ is analytic, and that $X=X_0\cup X_1$. Therefore, it remains to show that each $X_i$ is homogeneous.

So fix $i\in 2$. First observe that $X_i$ is dense in $X$. In particular, if $X$ is a meager space, then $X_i$ is also a meager space (see \cite[Exercise A.13.7]{van_mill_book}). On the other hand, if $X$ is a Baire space then $X$ has a dense Polish subspace by Lemma \ref{lemma_baire_dense_polish}, hence the same is true of $X_i$, because the $K_{n,i}$ are closed nowhere dense in $X$ (see \cite[Theorem 3.11]{kechris}).

By Theorem \ref{theorem_van_engelen} plus the previous paragraph, it will be enough to show that each $U_n\cap X_i$ contains a closed subspace homeomorphic to $X_i$. Let $X_i'$ be a subspace of $2^\omega$ such that $X_i'\approx X_i$, and observe that $X_i'\in\bS^1_1(2^\omega)$. It follows from our choice of $A$ that $X_i'\leq_\mathsf{W}A$. Therefore, by Lemma \ref{lemma_harrington}, this is witnessed by a continuous injection $f:2^\omega\longrightarrow 2^\omega$. Using the compactness of $2^\omega$, one sees that $f[X_i']$ is a closed subspace of $A$ homeomorphic to $X_i$. So $h_{n,i}[f[X_i']]$ is a closed subspace of $A_{n,i}$ homeomorphic to $X_i$ for each $n$. The observation that each $A_{n,i}=K_{n,i}\cap X_i$ is closed in $X_i$ concludes the proof.
\end{proof}

\begin{theorem}\label{theorem_main}
Let $X$ be a zero-dimensional analytic space. Then $X$ is $\sigma$-homogeneous with analytic witnesses.
\end{theorem}
\begin{proof}
Set
$$
\UU=\{U:U\text{ is a countable open subset of }X\},
$$
and observe that $C=\bigcup\UU$ is a countable open subset of $X$. Set $D=X\setminus C$. It is easy to check that either $D$ is empty or it is nowhere countable. Therefore, we will assume without loss of generality that $X$ is nowhere countable.

Now set
$$
\VV=\{V:V\text{ is an open meager subset of }X\},
$$
and observe that $M=\bigcup\VV$ is an open meager subset of $X$. It follows that $M$ is a meager space, and that $M$ is either empty or nowhere countable. Set $B=X\setminus M$. It is easy to check that $B$ is a Baire space, and that $B$ is either empty or nowhere countable. At this point, the desired result clearly follows from Lemma \ref{lemma_main}.
\end{proof}

By carefully considering the above proofs, one sees that every nowhere countable zero-dimensional analytic space can be written as the union of four homogeneous spaces (two for the meager part, and two for the Baire part). Therefore, one might wonder whether ever zero-dimensional analytic space can be written as the union of finitely many homogeneous spaces. In Section \ref{section_finite}, among other things, we will show that this is not the case (see Corollary \ref{corollary_scattered}).

\section{Pairwise disjoint witnesses}\label{section_disjoint}

The purpose of this section is to observe that a minor modification of the arguments given in Section \ref{section_main} yields $\sigma$-homogeneity with pairwise disjoint witnesses. The price to pay for this improvement is an increase in the complexity of the witnesses.

\begin{lemma}\label{lemma_main_disjoint}
Let $X$ be a zero-dimensional analytic space. Assume that $X$ is nowhere countable and that $X$ is either a meager space or a Baire space. Then $X$ is $\sigma$-homogeneous with pairwise disjoint $\bD^1_2$ witnesses.
\end{lemma}
\begin{proof}
We will use the same notation as in the proof of Lemma \ref{lemma_main}. In addition, set $B=2^\omega\setminus A$ and $B_{n,0}=h_{n,0}[B]$ for $n\in\omega$. Observe that
$$
X\setminus X_0=\bigcup_{n\in\omega}B_{n,0}
$$
is a coanalytic space. Therefore, it will be enough to show that $X\setminus X_0$ is homogeneous. This can be achieved as in the proof of Lemma \ref{lemma_main}, by applying Lemma \ref{lemma_harrington} to the reasonably closed Wadge class $\bGc=\bP^1_1(2^\omega)$.
\end{proof}

\begin{theorem}\label{theorem_main_disjoint}
Let $X$ be a zero-dimensional analytic space. Then $X$ is $\sigma$-homogeneous with pairwise disjoint $\bD^1_2$ witnesses.
\end{theorem}
\begin{proof}
Proceed as in the proof of Theorem \ref{theorem_main}, but apply Lemma \ref{lemma_main_disjoint} instead of Lemma \ref{lemma_main}.
\end{proof}

\section{Optimality}\label{section_optimality}

In this section, we will point out that the complexity of the witnesses to $\sigma$-homogeneity obtained in Section \ref{section_main} is as low as possible. In the case of Theorem \ref{theorem_main}, this follows from Theorem \ref{theorem_counterexample}. In the case of Theorem \ref{theorem_main_disjoint}, this follows from Corollary \ref{corollary_counterexample_disjoint}.

A weaker version of the following result is stated as part of \cite[Theorem 12.2]{medini_vidnyanszky}. However, it is easy to realize that the exact same proof (see \cite[Lemma 12.1]{medini_vidnyanszky}) actually yields the stronger result stated here.

\begin{theorem}[Medini, Vidny\'anszky]\label{theorem_counterexample}
Assume $\VL$. Then there exists a zero-dimensional analytic space that is not $\sigma$-homogeneous with coanalytic witnesses.	
\end{theorem}

Given a space $X=\bigcup_{n\in\omega}X_n$, where the $X_n$ are pairwise disjoint, it is clear that $X_n\in\bS^1_1(X)$ for each $n$ iff $X_n\in\bP^1_1(X)$ for each $n$. Combining this observation with Theorem \ref{theorem_counterexample} immediately yields the following result.

\begin{corollary}\label{corollary_counterexample_disjoint}
Assume $\VL$. Then there exists a zero-dimensional analytic space $X$ with the following properties:
\begin{itemize}
\item $X$ is not $\sigma$-homogeneous with pairwise disjoint analytic witnesses,
\item $X$ is not $\sigma$-homogeneous with pairwise disjoint coanalytic witnesses.
\end{itemize}
\end{corollary}

\section{Higher dimensions}\label{section_higher}

In this section, we will see that the results obtained so far, together with those of \cite{medini_vidnyanszky}, yield a complete picture of $\sigma$-homogeneity in the finite-dimensional realm. We begin by extending Theorems \ref{theorem_main} and \ref{theorem_main_disjoint}. The proof is essentially due to Ostrovsky, who applied the same argument in the Borel context (see \cite[page 663]{ostrovsky_sigma}). However, we do not know whether the assumption ``finite-dimensional'' can be dropped in the following theorem (see Question \ref{question_analytic}).

\begin{theorem}\label{theorem_main_generalized}
Let $X$ be a finite-dimensional analytic space. Then:
\begin{itemize}
\item $X$ is $\sigma$-homogeneous with analytic witnesses,
\item $X$ is $\sigma$-homogeneous with pairwise disjoint $\bD^1_2$ witnesses.
\end{itemize}
\end{theorem}
\begin{proof}
It is a fundamental result of dimension theory that every finite-dimensional space is homeomorphic to a subspace of $\RRR^n$ for some $n\in\omega$ (see for example \cite[Theorem 3.3.5]{van_mill_book}). Since $\QQQ$ and $\RRR\setminus\QQQ$ are both zero-dimensional and analytic, it follows that $X$ can be written as $X=\bigcup_{k\in n}X_k$ for some $n\in\omega$, where the $X_k$ are zero-dimensional, analytic, and pairwise disjoint. The desired results now follow from Theorems \ref{theorem_main} and \ref{theorem_main_disjoint}.
\end{proof}

Similarly, one can extend Theorems \ref{theorem_ostrovsky} and \ref{theorem_determinacy} as follows.

\begin{theorem}[Ostrovsky]\label{theorem_ostrovsky_generalized}
Every finite-dimensional Borel space is $\sigma$-homogeneous with pairwise disjoint $\mathsf{G}_\delta$ witnesses.
\end{theorem}

\begin{theorem}\label{theorem_determinacy_generalized}
Assume $\AD$. Then every finite-dimensional space is $\sigma$-homogeneous with pairwise disjoint $\mathsf{G}_\delta$ witnesses.
\end{theorem}

Given that Theorems \ref{theorem_ostrovsky} and \ref{theorem_determinacy} give closed witnesses, it is natural to wonder whether ``$\mathsf{G}_\delta$'' can be improved to ``closed'' in the above two theorems. The following proposition shows that this is not the case, and that $\mathsf{G}_\delta$ is the optimal complexity, even dropping the requirement that the witnesses are pairwise disjoint.

\begin{proposition}
There exists a one-dimensional compact space that is not $\sigma$-homogeneous with $\mathsf{F}_\sigma$ witnesses.
\end{proposition}
\begin{proof}
Let $P$ be the ``propeller space'' constructed by de Groot and Wille in \cite[Section 2]{de_groot_wille}.\footnote{\,For our purposes, we could have chosen every propeller to be a two-bladed propeller.} Assume that $P=\bigcup_{n\in\omega}P_n$, where each $P_n\in\bS^0_2(P)$. We will show that some $P_n$ is not homogeneous. Pick closed subsets $P_{n,k}$ of $P$ for $(n,k)\in\omega\times\omega$ such that each $P_n=\bigcup_{k\in\omega}P_{n,k}$. Since $P$ is a Baire space, we can fix $(n,k)\in\omega\times\omega$ and a non-empty open subset $U$ of $P$ such that $U\subseteq P_{n,k}$. Using the same notation as \cite{de_groot_wille}, one key property of $P$ is that if $p\in P\setminus\{a'_i:1\leq i<\omega\}$, then for every neighborhood $V$ of $p$ in $P$ there exists a neighborhood $V'\subseteq V$ of $p$ in $P$ such that $V'\setminus\{p\}$ is connected. On the other hand, no $a'_i$ has this property. Now fix $p\in U\setminus\{a'_i:1\leq i<\omega\}$ and $a'_j\in U$. Since $U$ is also open in $P_n$, it is clear that there can be no homeomorphism $h:P_n\longrightarrow P_n$ such that $h(p)=a'_j$. This shows that $P_n$ is not homogeneous, as desired.
\end{proof}

\section{Finite unions}\label{section_finite}

The purpose of this section is to distinguish $\sigma$-homogeneity from its finite analogues, and distinguish these finite analogues among themselves (see Corollary \ref{corollary_scattered}). Given $n$ such that $1\leq n\leq\omega$, we will say that a space $X$ is \emph{$n$-homogeneous} if $X=\bigcup_{k\in n}X_k$, where each $X_k$ is a homogeneous subspace of $X$. Similarly, we will say that a space $X$ is \emph{$n$-discrete} if $X=\bigcup_{k\in n}X_k$, where each $X_k$ is a discrete subspace of $X$.

\begin{theorem}\label{theorem_scattered}
Let $1\leq n\leq\omega$, and let $X$ be a scattered space of Cantor-Bendixson rank $n$. Then $X$ is $n$-discrete but not $\ell$-discrete for any $\ell<n$.
\end{theorem}
\begin{proof}
Set $X_k=X^{(k)}\setminus X^{(k+1)}$ for $k<n$, and observe that $\{X_k:k\in n\}$ is a partition of $X$ consisting of discrete spaces. So it remains to show that $X$ is not $\ell$-discrete for any $\ell<n$. First assume that $n<\omega$.

Given $k\in\omega$, we will denote by $\omega^{<k\uparrow}$ (respectively $\omega^{k\uparrow}$) the collection of all strictly increasing sequences of elements of $\omega$ of length less than $k$ (respectively of length exactly $k$). When $k=0$, we declare $\langle s_0,\ldots,s_{k-1}\rangle=\{s_0,\ldots,s_{k-1}\}=\varnothing$ and $s_{-1}=-1$. By starting with an arbitrary $f(\varnothing)\in X_{n-1}$ and continuing in the obvious way, one can construct a function $f:\omega^{<n\uparrow}\longrightarrow X$ satisfying the following conditions:
\begin{enumerate}
\item $f[\omega^{k\uparrow}]\subseteq X_{n-k-1}$ for every $k<n$,
\item $\langle f(s^\frown i):s_{k-1}<i<\omega\rangle$ converges to $f(s)$ for all $s=\langle s_0,\ldots,s_{k-1}\rangle\in\omega^{<n\uparrow}$.
\end{enumerate}

Now assume, in order to get a contradiction, that $X$ is $\ell$-discrete for some $\ell<n$. Write $X=\bigcup_{i\in\ell}D_i$, where the $D_i$ are pairwise disjoint discrete subspaces of $X$. Define the coloring $c:[\omega]^{<n}\longrightarrow\ell$ by setting
$$
c(\{s_0,\ldots,s_{k-1}\})=i\text{ iff }f(\langle s_0,\ldots,s_{k-1}\rangle)\in D_i,
$$
where we are assuming that $s_0<\cdots <s_{k-1}$. By repeatedly applying the classical Ramsey Theorem (see for example \cite[Theorem 9.1]{jech}), one obtains sets $H_k$ for $k<n$ satisfying the following conditions:
\begin{itemize}
\item $\omega=H_0\supseteq\cdots\supseteq H_{n-1}$,
\item $|c[[H_k]^k]|=1$ for each $k$.
\end{itemize}
Denote by $i_k$ the unique element of $c[[H_k]^k]$ for $k<n$. By the Pigeonhole Principle, we can fix $i<\ell$ and $j,k<n$ such that $j<k$ and $i_j=i_k=i$. Also fix $s_0,\ldots,s_{j-1}\in H_k\subseteq H_j$ such that $s_0<\cdots <s_{j-1}$, and set $s=\langle s_0,\ldots,s_{j-1}\rangle$. It is clear that $f(s)\in D_i$. Now set
$$
S=\{f(s^\frown\langle s_j,\ldots ,s_{k-1}\rangle):s_j,\ldots ,s_{k-1}\in H_k\text{ and }s_{j-1}<s_j<\cdots <s_{k-1}\},
$$
and observe that $S\subseteq D_i$. Using condition (2), it is easy to realize that $f(s)$ belongs to the closure of $S$ in $X$. Furthermore, condition (1) guarantees that $f(s)\notin S$. This contradicts the assumption that $D_i$ is discrete.

Finally, assume that $n=\omega$. If $X$ were $\ell$-discrete for some $\ell<n$, then so would be $X\setminus X^{(\ell+1)}$. But this would contradict the first part of this proof, since $X\setminus X^{(\ell+1)}$ is a scattered space of Cantor-Bendixson rank $\ell+1$.
\end{proof}

\begin{corollary}\label{corollary_scattered}
Let $1\leq n\leq\omega$. Then there exists a zero-dimensional scattered space that is $n$-homogeneous but not $\ell$-homogeneous for any $\ell<n$.
\end{corollary}
\begin{proof}
Consider the ordinal $\omega^n$ with the order topology. It is not hard to check that this space is scattered of Cantor-Bendixson rank $n$. The desired result then follows from Theorem \ref{theorem_scattered} plus the observation that a scattered space is homogeneous iff it is discrete.
\end{proof}

\section{Open questions}

The most pressing open question is the following. We remark that we would not know the answer even if we substituted ``analytic'' with much more restrictive conditions, such as ``compact''.

\begin{question}\label{question_analytic}
Is every analytic space $\sigma$-homogeneous?
\end{question}

Recall that a space is \emph{countable-dimensional} if it can be written as a countable union of zero-dimensional spaces. Notice that, by the $\mathsf{G}_\delta$-Enlargement Theorem of dimension theory (see for example \cite[Corollary 3.3.12]{van_mill_book}), these zero-dimensional subspaces can be chosen to be $\mathsf{G}_\delta$. Using this observation, it is straightforward to see that Theorem \ref{theorem_main_generalized} holds for all countable-dimensional spaces. The same is true of Theorems \ref{theorem_ostrovsky_generalized} and \ref{theorem_determinacy_generalized}, except that we do not know whether one can still obtain pairwise disjoint witnesses (without increasing their complexity beyond $\mathsf{G}_\delta$).

\begin{question}
Do Theorems \ref{theorem_ostrovsky_generalized} and \ref{theorem_determinacy_generalized} hold for all countable-dimensional spaces?
\end{question}

Given the results of Section \ref{section_disjoint}, it seems fitting to mention that the following question is still open (see \cite[Question 13.5]{medini_vidnyanszky}).

\begin{question}[Medini, Vidny\'anszky]
In $\ZFC$, is there a zero-dimensional $\sigma$-homogeneous space that is not $\sigma$-homogeneous with pairwise disjoint witnesses? At least under additional set-theoretic assumptions?
\end{question}

Finally, we remark that the $\sigma$-homogeneity of the zero-dimensional spaces considered in this article (as well as in \cite{medini_vidnyanszky} and \cite{ostrovsky_sigma}) can always be witnessed by \emph{strongly} homogeneous subspaces. This observation suggests the following question.

\begin{question}
In $\ZFC$, is there a zero-dimensional $\sigma$-homogeneous (or even homogeneous) space that is not $\sigma$-homogeneous with strongly homogeneous witnesses? At least under additional set-theoretic assumptions?
\end{question}

\end{document}